\documentclass{amsart}

\usepackage{sets,defs,words,letters}

\newtheorem{theorem}{Theorem}
\newtheorem*{theorem*}{Theorem}

\newtheorem*{proposition*}{Proposition}
\newtheorem{proposition}{Proposition}
\newtheorem{cortoprop}[proposition]{Corollary}
\newtheorem*{corollary*}{Corollary}
\newtheorem{corolemma}[proposition]{Lemma}

\newtheorem*{definition*}{Definition}
\newtheorem*{lemma*}{Lemma}
\newtheorem{lemma}[proposition]{Lemma}
\theoremstyle{remark}
\newtheorem*{remark}{Remark}

\numberwithin{proposition}{section}
\numberwithin{equation}{section}

\begin{document}

\title{The deformation theory of sheaves of commutative rings II}
\date{\today}
\author{Jonathan Wise}
\email{jonathan@math.ubc.ca}
\thanks{Research partially supported by NSF-MSPRF 0802951}
\address{Department of Mathematics,
The University of British Columbia,
Room 121, 1984 Mathematics Road,
Vancouver, B.C.,
Canada V6T 1Z2}

\begin{abstract}
  In this sequel to \cite{def-rings}, we compare the obstruction classes defined there to those defined by Illusie.  We also give sheaf theoretic proofs of the standard properties of the cotangent complex.
\end{abstract}

\maketitle

\section{Introduction}
\label{sec:intro}

Following a line of thought initiated by Quillen~\cite{Q} and Gaitsgory~\cite{Gaits}, we showed in~\cite{def-rings} that obstructions to infinitesimal deformations of sheaves of commutative rings and their homomorphisms arise as the classes of torsors and banded gerbes naturally associated to those problems.  We also showed that the obstruction \emph{groups}, in which the obstructions lie, coincide with those defined by Illusie~\cite{Illusie}.  However, we left for the present paper the question of whether the obstruction \emph{classes} defined in~\cite{def-rings} agree with those defined in~\cite{Illusie}.  We shall answer this question in the affirmative in Section~\ref{sec:obs}.

Illusie's definition of the obstruction classes relies on the connecting homomorphism in the \emph{transitivity sequence} \cite[II.2.1]{Illusie} associated to a sequence of ring homomorphisms $A \rightarrow B \rightarrow C$.  Our comparison relies on an interpretation of the connecting homomorphisms, in terms of torsors and banded gerbes, that is proved in Section~\ref{sec:transitivity}.  Along the way, we obtain a proof of the exactness of the transitivity sequence that does not rely on simplicial homotopy theory.

Having studied one of the two fundamental properties of the cotangent complex without simplicial homotopy theory, we could hardly stop without considering the base change theorem \cite[II.2.2]{Illusie}.  In Section~\ref{sec:base-change}, we give a proof of this theorem using the \v{C}ech spectral sequence \cite[V.3.3]{sga4-2}.  (It is a matter of opinion whether on the convergence of the \v{C}ech spectral sequence can qualify as avoiding simplicial homotopy theory; at least from a certain point of view it may be simpler than the monadic approach employed by Illusie \cite[I.5]{Illusie}.)

I am grateful to Dan Abramovich for catching several mistakes and suggesting a number of clarifications.

\section{Notation and review of \cite{def-rings}}
\label{sec:notation}

All rings and algebras will be assumed commutative and unital.

Suppose that $A$ is a sheaf of rings on a topos $E$ and $B$ is an $A$-algebra.  In \cite{def-rings}, we defined a site $A\uAlg(E)/B$ (or $A\uAlg/B$ when the dependence on $E$ is clear) whose objects are pairs $(U, C)$ where $U$ is an object of $E$ and $C$ is an $A_U$-algebra with a map to $B_U$.  This site is equipped with a topology \cite[Section~5]{def-rings} in which all algebras are covered by finitely generated polynomial algebras over $A$.  Since the standard deformation problems associated to rings are trivial for polynomial algebras, this topology enables one to say that the standard deformation problems are locally trivial; since solutions to the standard deformation problems can be glued in this topology, it follows that obstructions to the existence of global solutions arise from the failure of local solutions to glue \cite[Sections~6 and~7]{def-rings}.  In short, solutions form torsors and gerbes and the cohomology classes of these gerbes serve as obstructions.

If $J$ is a $B$-algebra, the sheaf $\uDer_A(B, J)$ plays a role of prime importance in the theory.  Its sections over $(U, C)$ are $A_U$-derivations from $C$ to $J_U$, where $J_U$ is given the $C$-module structure inherited from the structural morphism $C \rightarrow B_U$.  Note that $\uDer_A(B, J)$ is represented by the $A$-algebra $B + J$, the trivial square-zero extension of $B$ with ideal $J$.

In this paper, we will need to construct a number of morphisms of sites and topoi, and we will make some abuses of language systematically.  If $u^\ast : E \rightarrow E'$ is a left exact continuous functor between sites, then it induces a morphism of sites $u : E' \rightarrow E$.  We will reuse the symbol $u^\ast$ to denote the pullback functor on the categories of sheaves.  Its right adjoint will always be denoted $u_\ast$, and, should one exist, its left adjoint will be $u_!$.

\section{The transitivity sequence}
\label{sec:transitivity}

Define adjoint functors
\begin{alignat*}{3}
  \pi&_!  &&: A \uAlg(E)/B \rightarrow E : (U,C) \mapsto U\\
  \pi&^\ast&&  : E \rightarrow A\uAlg(E)/B : U \mapsto (U, B_U) 
\end{alignat*}
with $\pi_!$ left adjoint to $\pi^\ast$.  Because $\pi^\ast$ is left exact and continuous (or because $\pi_!$ is cocontinuous) these determine a morphism of sites
\begin{equation*}
  \pi : A\uAlg(E)/B \rightarrow E .
\end{equation*}
When it is necessary to emphasize the dependence of $\pi$ on $A$ and $B$, we shall write $\pi^{B/A}$ in lieu of $\pi$.

The purpose of this section is to prove the following theorem and relate it to \cite[Proposition~2.1.2]{Illusie}.
\begin{theorem} \label{thm:trans}
  Suppose that $A \rightarrow B \rightarrow C$ is a sequence of homomorphisms of sheaves of algebras and $J$ is a $C$-module.  Then there is an exact triangle in the derived category of sheaves of $C$-modules,
  \begin{equation*}
    R \pi^{C/B}_\ast \uDer_B(C,J) \rightarrow R \pi^{C/A}_\ast \uDer_A(C,J) \rightarrow R \pi^{B/A}_\ast \uDer_A(B,J) .
  \end{equation*}
\end{theorem}

\subsection{Constant sheaves}

Notice that the functor $\pi_! : A\uAlg/B \rightarrow E$ is exact, so its extension to sheaves is exact as well (by \cite[Proposition~I.5.4~4]{sga4-1}, coupled with the exactness of sheafification), so $\pi^\ast$ preserves injectives.  From this we deduce 
\begin{proposition} \label{prop:const-acyclic}
  For any sheaf of abelian groups $J$ on $E$, the map $J \rightarrow R \pi_\ast \pi^\ast J$ is a quasi-isomorphism.
\end{proposition}

Since $\pi^\ast$ preserves injectives, it is sufficient to demonstrate that $J \rightarrow \pi_\ast \pi^\ast J$ is an isomorphism.  For this, note that
\begin{equation*}
  \Gamma(U, \pi_\ast \pi^\ast J) = \Gamma(\pi^\ast U, \pi^\ast J) = \Gamma(\pi_! \pi^\ast U, J) = \Gamma(U, J) 
\end{equation*}
since $\pi_! \pi^\ast U = U$.  \qed

\subsection{An intermediate site}

Suppose that $A \rightarrow B \rightarrow C$ is a sequence of homomorphisms of rings of $E$.  If $J$ is a $C$-module, then the homomorphism $B \rightarrow C$ also gives it the structure of a $B$-module.  We may therefore define sheaves
\begin{alignat*}{2}
  &\uDer_A(B,J)  & \quad &\text{on $A\uAlg/B$,} \\
  &\uDer_A(C,J)  & &\text{on $A\uAlg/C$, and} \\
  &\uDer_B(C,J)  & &\text{on $B\uAlg/C$.}
\end{alignat*}
In order to relate the cohomology groups of these sheaves, we shall define a fourth site $A\uAlg/BC$ and several sheaves on it that are closely related to the sheaves enumerated above.

Let $A\uAlg/BC = A\uAlg(E)/BC$ be the category of all $(U, B'C')$, where $B'C'$ is shorthand notation for a diagram of $A_U$-algebras,
\begin{equation*}
  \xymatrix{
    B' \ar[r] \ar[d] & C' \ar[d] \\
    B_U \ar[r] & C_U .
  }
\end{equation*}
We shall call a family of maps $(U_i, B'_i C'_i) \rightarrow (U, B'C')$ covering if, for any $V$ over $U$, and any finite collections $S \subset \Gamma(U, B')$ and $T \subset \Gamma(U, C')$, it is possible, locally in $V$, to find a factorization of $V \rightarrow U$ through some $U_i$ and find lifts of $S$ to $\Gamma(V, B'_i)$ and $T$ to $\Gamma(V, C'_i)$.

\begin{remark}
  This site bears some formal resemblance to one arising from Illusie's mapping cylinder topos \cite[III.4]{Illusie} associated to the morphism of sites $u : A\uAlg/B \rightarrow A\uAlg/C$ (Section~\ref{sec:ex-tri}).  However, this resemblance appears to be superficial, since the arrows in $A\uAlg/BC$ are directed opposite those in the mapping cylinder topos.  It would be interesting, though, to know if the exactness of the transitivity triangle could be proved using a mapping cylinder construction in place of $A\uAlg/BC$.
\end{remark}

\begin{lemma} \label{lem:loc-triv}
  If $B$ is a finitely generated polynomial algebra over $A$, and $C$ is a finitely generated polynomial algebra over $B$, then all covers of $BC$ in $A\uAlg(E)/BC$ are pulled back from covers of the final object in $E$.
\end{lemma}

In the definition of the topology on $A\uAlg/BC$ above, take $S$ to be the set of generators of $B$ over $A$, and take $T$ to be the set of generators for $C$ over $B$.  \qed

{\allowdisplaybreaks
There are a number of functors relating $A\uAlg/BC$ to other sites:
\begin{alignat*}{7}
  \sigma&_! & && \quad : \quad && A\uAlg/&B  &\rightarrow  A\uAlg/&BC& \quad : & \quad  &B' &\mapsto B'B' \\
  \sigma&^\ast& = \alpha&_!  &:\quad&&A\uAlg/&BC &\rightarrow  A\uAlg/&B  & : & \quad & B'C' &\mapsto B' \\
  \sigma&_\ast &= \alpha&^\ast &:\quad&& A\uAlg/&B  &\rightarrow A\uAlg/&BC  & :& &B'  &\mapsto B'C \\
  \\
  \tau&_!&&&:\quad&&A\uAlg/&C  &\rightarrow A\uAlg/&BC&:&&C' & \mapsto AC' \\
  \tau&^\ast& = \beta&_!  &:\quad&&A\uAlg/&BC & \rightarrow A\uAlg/&C & : &&B' C' & \mapsto C' \\
  \tau&_\ast& = \beta&^\ast&:\quad&& A\uAlg/&C & \rightarrow A\uAlg/&BC&: &&C'&\mapsto (C'\fp_C B)C' \\
  \\
  &&\gamma&_! &:\quad&&A\uAlg/&BC &\rightarrow B\uAlg/&C &:&& B'C'& \mapsto C' \tensor_A B \\
  &&\gamma &^\ast&:\quad& &B\uAlg/&C & \rightarrow A\uAlg/&BC&:&&C'&\mapsto BC' .
\end{alignat*}
In each case, the functor $F_!$ is left adjoint to $F^\ast$.  This implies that the functors $\alpha^\ast$, $\beta^\ast$, and $\gamma^\ast$ are all left exact, and that $\alpha$ and $\beta$ both have exact left adjoints.  Since $\alpha^\ast$, $\beta^\ast$, $\gamma^\ast$, $\sigma^\ast$, and $\tau^\ast$ are all continuous we obtain morphisms of topoi:
\begin{alignat*}{1}
  \alpha &  :  A\uAlg/BC  \rightarrow   A\uAlg/B \\
  \beta & : A\uAlg/BC  \rightarrow  A \uAlg/C \\
  \gamma & : A\uAlg/BC   \rightarrow  B\uAlg/C \\ 
  \sigma & : A\uAlg/B  \rightarrow  A\uAlg/BC \\
  \tau & : A\uAlg/C  \rightarrow  A\uAlg/BC .
\end{alignat*}
and the exact left adjoints $\alpha_!$ and $\beta_!$ to $\alpha^\ast$ and $\beta^\ast$, defined originally on the level of sites, extend to exact left adjoints on the level of sheaves \cite[Proposition~I.5.4~4]{sga4-1}.  Since the natural transformations $\alpha_! \alpha^\ast \rightarrow \id$ and $\beta_! \beta^\ast \rightarrow \id$ are isomorphisms, so are the natural transformations $\id \rightarrow \alpha_\ast \alpha^\ast$ and $\id \rightarrow \beta_\ast \beta^\ast$ \cite[Proposition~I.5.6]{sga4-1} (or mimic the proof of Proposition~\ref{prop:const-acyclic}).  As in Proposition~\ref{prop:const-acyclic}, it follows from this that for any sheaf of abelian groups $F$ on $A\uAlg/B$ (resp.\ on $A\uAlg/C$), we have}
\begin{multline}
  H^\ast\bigl(A\uAlg/BC, \alpha^\ast F\bigr) = H^\ast(A\uAlg/B, F)  \label{eqn:11} \\ (\text{resp.\ } H^\ast\bigl(A\uAlg/BC,\beta^\ast F\bigr) = H^\ast(A\uAlg/C, F) \: ) .
\end{multline}

\subsection{Sheaves on $A\uAlg/BC$}
\label{sec:sheaves}

Suppose that $I$ is a $B$-module, $J$ is a $C$-module, and $I \rightarrow J$ is a $B \rightarrow C$ homomorphism of modules.  Then we shall say that $IJ$ is a $BC$-module.

If $IJ$ is a $BC$-module, define $\uDer_A(BC, IJ)$ to be the sheaf on $A\uAlg/BC$ represented by the object
\begin{equation*}
  \xymatrix{
    B + I \ar[r] \ar[d] & C + J \ar[d] \\
    B \ar[r] & C .
  }
\end{equation*}
A section of $\uDer_A(BC, IJ)$ over $B'C'$ is therefore a pair of derivations $B' \rightarrow I$ and $C' \rightarrow J$ that induce the same derivation $B' \rightarrow J$.

We may make the following identifications:
\begin{alignat*}{2}
  \alpha&^\ast \uDer_A(B, J) &= \: &  \uDer_A(BC, J0) \\
  \beta&^\ast \uDer_A(C,J) &= \: & \uDer_A(BC, JJ) \\
  \gamma&^\ast \uDer_B(C,J) &= \: &\uDer_A(BC, 0J) .
\end{alignat*}
Combining these with Equations~\eqref{eqn:11} demonstrates that
\begin{alignat}{2} \label{eqn:2}
  H^\ast\bigl(A\uAlg/BC, \uDer_A(BC,J0)\bigr) & = H^\ast\bigl(A\uAlg/B, \uDer_A(B,J)\bigr) \\
  H^\ast\bigl(A\uAlg/BC, \uDer_A(BC,JJ)\bigr) & = H^\ast\bigl(A\uAlg/C, \uDer_A(C,J)\bigr) \notag.
\end{alignat}

\begin{proposition}
  The sequence
  \begin{equation} \label{eqn:3}
    0 \rightarrow 0J \rightarrow JJ \rightarrow J0 \rightarrow 0
  \end{equation}
  of $BC$-modules is exact.  The induced sequence
  \begin{equation} \label{eqn:1}
    0 \rightarrow \uDer_A(BC, 0J) \rightarrow \uDer_A(BC, JJ) \rightarrow \uDer_A(BC, J0) \rightarrow 0  
  \end{equation}
  is also exact.
\end{proposition}

The exactness of~\eqref{eqn:3} is immediate.  Left exactness of~\eqref{eqn:1} can be checked directly.  It also follows from the existence of a left adjoint for the functor sending $IJ$ to $(B + I)(C + J)$:  the left adjoint sends $BC$ to $\Omega_B \Omega_C$.

To prove the surjectivity of $\uDer_A(BC,JJ) \rightarrow \uDer_A(BC,J0)$ we must show that any derivation $B \rightarrow J$ may be extended to a derivation $C \rightarrow J$, locally in $A\uAlg/BC$.  Since this is a local question, we are free to assume that $C$ is a finitely generated polynomial algebra over $B$, and then the assertion is obvious.  \qed

\begin{remark}
  Note the significance of the topology in the proof:  the sequence is far from exact on the level of presheaves.
\end{remark}

\begin{proposition} \label{prop:gamma}
  The natural map $\uDer_B(C,J) \rightarrow R \gamma_\ast \uDer_A(BC, 0J)$ is an isomorphism.
\end{proposition}

This is a local question on $B\uAlg/C$, so we can assume that $C = B[S]$ for some finite set $S$.  Now consider the functor $A\uAlg/B \rightarrow A\uAlg/BC$ sending $B'$ to the diagram
\begin{equation*}
  \xymatrix{
    B' \ar[r] \ar[d] & B'[S] \ar[d] \\
    B \ar[r] & C .
  }
\end{equation*}
This is left exact and continuous, so it induces a morphism of sites
\begin{equation*}
  t : A\uAlg/BC \rightarrow A\uAlg/B .
\end{equation*}

\begin{lemma}
  The morphism $t$ is acyclic.
\end{lemma}

Locally in $A\uAlg/B$, the $A$-algebra $B$ is freely generated by a finite set $T$.  The corresponding object $t^\ast B$ is
\begin{equation*}
  \xymatrix{
    A[T] \ar[r] \ar[d] & A[T \amalg S] \ar[d] \\
    B \ar[r] & C ,
  }
\end{equation*}
which is acyclic relative to $E$ in $A\uAlg/BC$ by Lemma~\ref{lem:loc-triv}.  \qed

The proposition therefore reduces to showing that $t_\ast \uDer_A(BC, 0J)$ is acyclic relative to $E$ when $C = B[S]$.  But in that case, $t_\ast \uDer_A(BC, 0J) = \pi^\ast J^S$, where $\pi : A\uAlg/B \rightarrow E$ is the projection.  By Proposition~\ref{prop:const-acyclic}, the sheaf $\pi^\ast J$ is acyclic relative to $E$.  \qed

Let $\pi$ denote the projection from $A\uAlg/BC$ to $E$.  Applying $R\pi_\ast$ to the exact sequence~\eqref{eqn:1} gives an exact triangle on $E$, and in view of the identifications of~\eqref{eqn:2} and the proposition, we get an exact triangle
\begin{equation} \label{eqn:4}
  R \pi^{C/B}_\ast \uDer_B(C,J) \rightarrow R \pi^{C/A}_{\ast} \uDer_A(C,J) \rightarrow R \pi^{B/A}_\ast \uDer_A(B,J)
\end{equation}
in the derived category of sheaves of $C$-modules on $E$, completing the proof of Theorem~\ref{thm:trans}.  It follows that we have a long exact sequence
\begin{multline} \label{eqn:5}
  \cdots \rightarrow R^p \pi^{C/B}_\ast \uDer_B(C,J) \rightarrow R^p \pi^{C/A}_\ast \uDer_A(C,J) \rightarrow R \pi^{B/A}_\ast \uDer_A(B,J) \\ \rightarrow R^{p+1} \pi^{C/B}_\ast \uDer_B(C,J) \rightarrow \cdots .
\end{multline}
Pushing forward to a point, we also get a long exact sequence (the ``deformation--obstruction sequence''),
\begin{multline*}
  \cdots \rightarrow H^p\bigl(B\uAlg/C, \uDer_B(C,J)\bigr) \rightarrow H^p\bigl(A\uAlg/C, \uDer_A(C,J)\bigr) \\ \rightarrow H^p\bigl(A\uAlg/B, \uDer_A(B,J)\bigr) \rightarrow H^{p+1}\bigl(B\uAlg/C, \uDer_B(C,J)\bigr) \rightarrow \cdots  .
\end{multline*}

\subsection{Computing the exact triangle}
\label{sec:ex-tri}

The maps in the exact triangle~\eqref{eqn:4} were defined using the site $A\uAlg/BC$.  In this section we will see how these maps can be defined more intrinsically in terms of the algebra homomorphisms $A \rightarrow B \rightarrow C$.  

There is a functor $u_! : A\uAlg/B \rightarrow A\uAlg/C$ sending an $A$-algebra over $B$ to the $A$-algebra over $C$ obtained by composition with $B \rightarrow C$.  This has the right adjoint $u^\ast : A\uAlg/C \rightarrow A\uAlg/B$ sending $C'$ to $C' \fp_C B$.  This functor is left exact and continuous and gives rise to a morphism of sites $u : A \uAlg/B \rightarrow A\uAlg/C$.  The fuctor $u^\ast$ factors through $A\uAlg/BC$ as $\sigma^\ast \circ \beta^\ast$, giving a facorization of the morphism of sites $u : A\uAlg/B \rightarrow A\uAlg/C$ as $\beta \circ \sigma$.

There is also a functor $v_! : A\uAlg/C \rightarrow B\uAlg/C$ sending an $A$-algebra $C'$ over $C$ to the $B$-algebra $B \tensor_A C'$ over $C$.  It has the right adjoint $v^\ast$ which gives a $B$-algebra over $C$ its $A$-algebra structure induced from the map $A \rightarrow B$.  Since $v^\ast$ is exact and continuous, it induces a morphism of sites $v : A\uAlg/C \rightarrow B\uAlg/C$.  The factorization $v^\ast = \tau^\ast \circ \gamma^\ast$ permits us to factor $v = \gamma \circ \tau$.

On $A\uAlg/C$ there is a commutative diagram
\begin{equation*}
  \xymatrix{
    F \ar[r] \ar[d] & R u_\ast u^\ast F \ar@{=}[d] \\
    R \beta_\ast \beta^\ast F \ar[r] \ar[ur] & R \beta_\ast \alpha^\ast u^\ast F .
  }
\end{equation*}
The triangle in the upper left merely expresses the fact that the unit of the adjunction $(u^\ast, u_\ast)$ is the composition of the units of the adjunctions $(\beta^\ast, \beta_\ast)$ and $(\sigma^\ast, \sigma_\ast)$.  The identification on the right side recognizes that $u_\ast = \beta_\ast \sigma_\ast = \beta_\ast \alpha^\ast$.

Applying $R \pi^{C/A}_\ast$ to the commutative diagram above, we get the commutative diagram
\begin{equation}\label{eqn:12}
  \xymatrix{
    R \pi^{C/A}_\ast F \ar[r] \ar[d] & R \pi^{B/A}_\ast u^\ast F \ar@{=}[d] \\
    R \pi_\ast \beta^\ast F \ar[r] & R \pi_\ast \alpha^\ast u^\ast F
  }
\end{equation}
in the derived category of sheaves of $C$-modules on $E$.  (Here $\pi$ denotes the projection  $A\uAlg/BC \rightarrow E$.)  The vertical arrow on the left side of the diagram was proved to be an isomorphism in Section~\ref{sec:sheaves}.

\begin{lemma}
  The diagram
  \begin{equation*}
    \xymatrix@R=10pt{
      R \pi^{C/B}_\ast \uDer_B(C,J) \ar[r] & R \pi^{C/A}_\ast \uDer_A(C,J) \ar[r] & R \pi^{B/A}_\ast \uDer_A(B,J) \\
      R \pi_\ast \uDer_A(BC,0J) \ar[r] \ar@{=}[u] & R \pi_\ast \uDer_A(BC,JJ) \ar@{=}[u] \ar[r] & R \pi_\ast \uDer_A(BC,J0) \ar@{=}[u] 
    }
  \end{equation*}
  commutes.  The vertical identifications are those of Equations~\eqref{eqn:2} and Proposition~\ref{prop:gamma};  the upper row is the exact triangle~\eqref{eqn:4}; the morphisms in the lower row are induced from the morphisms of sites $A\uAlg/B \xrightarrow{u} A\uAlg/C \xrightarrow{v} B\uAlg/C$.
\end{lemma}

The commutativity of the square on the right is Diagram~\eqref{eqn:12}, applied with $F = \uDer_A(C,J)$.

The commutativity of the square on the left can be obtained in a very similar way, using $v$ in place of $u$.  In this case, we apply have a commutative diagram
\begin{equation*}
  \xymatrix{
    F \ar[r] \ar[d] & R v_\ast v^\ast F \ar[d] \\
    R \gamma_\ast \gamma^\ast F \ar[r] \ar[ur] & R \gamma_\ast \beta^\ast v^\ast F
  }
\end{equation*}
for any sheaf $F$ on $B\uAlg/C$.  Applying $R \pi^{C/B}_\ast$ and substituting $F = \uDer_B(C, J)$ gives the commutativity of the square on the left.  \qed

\subsection{The cotangent complex}

We continue to consider a sequence $A \rightarrow B \rightarrow C$ of homomorhpisms of algebras of $E$.

Recall that there are canonical morphisms of cotangent complexes \cite[II.2]{Illusie},
\begin{equation} \label{eqn:13}
  \bL_{B/A} \tensor_B C \rightarrow \bL_{C/A} \rightarrow \bL_{C/B} .
\end{equation}
In fact, this is an exact triangle \cite[Proposition~II.2.1.2]{Illusie}; we shall provide an alternate proof of this fact by relating the triangle above to the exact triangle of Theorem~\ref{thm:trans}.

The following construction of the morphisms of the sequence~\eqref{eqn:13} is compatible with that given by Illusie \cite[II.1.2.3]{Illusie}.  Let $P_A(B)$ \cite[II.1.2.1]{Illusie} be the standard simplicial resolution of $B$ by free $A$-algebras, with $P_A(B)_0 = A[B]$ and $P_A(B)_{n+1} = A[P_A(B)_n]$.  Then there is a canonical map $P_A(B) \rightarrow P_A(C)$ such that the diagram
\begin{equation*}
  \xymatrix{
    P_A(B) \ar[r] \ar[d] & P_A(C) \ar[d] \\
    B \ar[r] & C
  }
\end{equation*}
commutes.  This induces a $P_A(B) \rightarrow P_A(C)$ homomorphism $\Omega_{P_A(B)/A} \rightarrow \Omega_{P_A(C)/A}$, hence a $B \rightarrow C$ homomorphism $\bL_{B/A} \rightarrow \bL_{C/A}$, and finally the desired map $\bL_{B/A} \tensor_B C \rightarrow \bL_{C/A}$.

For the second map, note that there is a canonical homomorphism of simplicial $A$-algebras $P_A(C) \rightarrow P_B(C)$, defined in degree zero by $A[C] \rightarrow B[C]$ and then by induction in degree $n + 1$, by $A[P_A(C)_n] \rightarrow B[P_B(C)_n]$.  This induces a homomorphism $\Omega_{P_A(C)} \rightarrow \Omega_{P_B(C)}$, compatible with the map $P_A(C) \rightarrow P_B(C)$, and hence a map of complexes of $C$-modules, $\bL_{B/A} \rightarrow \bL_{C/A}$.

Our task in this section is to prove

\begin{proposition} \label{prop:bdry-map}
  The triangle~\eqref{eqn:13} represents the triangle~\eqref{eqn:4} by application of the functor $R\uHom(-, J)$.
\end{proposition}

We must show that the following diagrams commute:
\begin{gather*}
  \xymatrix@C=15pt@R=10pt{
    R \uHom_C(\bL_{C/B}, J) \ar[r] \ar@{=}[d] & R \uHom_C(\bL_{C/A}, J) \ar[r] \ar@{=}[d] & R \uHom_B(\bL_{B/A}, J) \ar@{=}[d]  \\
    R \pi^{C/B}_\ast \uDer_A(C, J) \ar[r] & R \pi^{C/A}_\ast \uDer_A(C, J) \ar[r] & R \pi^{B/A}_\ast \uDer_A(B,J) .
  } \\ \\
  \xymatrix@C=15pt@R=10pt{
    R \uHom_B(\bL_{B/A}, J) \ar@{=}[d] \ar[r] & R \uHom_C(\bL_{C/B}, J)[1] \ar@{=}[d] \\
    R \pi^{B/A}_\ast \uDer_A(B,J) \ar[r] & R \pi^{C/B}_\ast \uDer_A(C,J)[1] .
  }
\end{gather*}
The vertical arrows were shown to be equivalences in \cite[Theorem~3]{def-rings}.  Note that $R \uHom_B(\bL_{B/A}, J) = R \uHom_C(\bL_{B/A} \tensor_B C, J)$.

Since any sheaf of $C$-modules $J$ can be resolved by injectives, it is sufficient to prove the proposition when $J$ is injective.  In this case, we know that the hypercover $P_A(B) P_A(C)$ of $A\uAlg/BC$ is acyclic for the sheaves $\uDer_A(BC, J0)$ and $\uDer_A(BC, JJ)$ and that the map $\uDer_A(BC, JJ) \rightarrow \uDer_A(BC, J0)$ induces the second arrow of~\eqref{eqn:4}.   On the other hand, evaluating
\begin{equation*}
  \Gamma\bigl(P_A(B) P_A(C), \uDer_A(BC,JJ)\bigr) \rightarrow \Gamma\bigl(P_A(B) P_A(C), \uDer_A(BC, J0)\bigr)
\end{equation*}
gives precisely the map
\begin{equation*}
  \Hom(\bL_{B/A} \tensor_B C, J) \rightarrow \Hom(\bL_{C/A}, J) .
\end{equation*}

The proof for the first arrow of~\eqref{eqn:4} is very similar:  $P_A(C) P_B(C)$ is acyclic for both $\uDer_A(BC,JJ)$ and $\uDer_A(BC,0J)$, so the map
\begin{equation*}
  \Gamma\bigl(P_A(C) P_B(C), \uDer_A(BC,0J)\bigr) \rightarrow \Gamma\bigl(P_A(C) P_B(C), \uDer_A(BC,JJ)\bigr)
\end{equation*}
induces the first arrow of~\eqref{eqn:4} and can be evalued directly to give the map
\begin{equation*}
  \Hom(\bL_{C/B}, J) \rightarrow \Hom(\bL_{C/A}, J) .
\end{equation*}

We must also check that the boundary maps
\begin{equation*}
  R \pi_\ast^{B/A} \uDer_A(B,J) \rightarrow R \pi_\ast^{C/B} \uDer_B(C,J) [1]
\end{equation*}
agree with the maps
\begin{equation*}
  R \uHom(\bL_{B/A}, J) \rightarrow R \uHom(\bL_{C/B}, J) [1]
\end{equation*}
defined by Illusie.  Let $P = P_A(B)$ be the standard simplicial resolution of $B$.  Let $Q = P_P^\Delta(C)$ \cite[II.1.2.2.1]{Illusie}.  Then $Q$ is a simplicial resolution of $C$ that is free, term by term, over $P$.  The simplicial object $PQ$ of $A\uAlg/BC$ is a hypercover, and if $J$ is injective then the sheaves $\uDer_A(BC, J0)$, $\uDer_A(BC, JJ)$, and $\uDer_A(BC, 0J)$ are all acyclic for $PQ$.  Evaluating the exact sequence~\eqref{eqn:1} on $PQ$, we obtain an exact sequence of complexes (of sheaves on $E$),
\begin{equation*}
  0 \rightarrow \pi^{BC/A}_\ast \uDer_A(PQ, 0J) \rightarrow \pi^{C/A}_\ast \uDer_A(Q, J) \rightarrow \pi^{B/A}_\ast \uDer_A(P, J) \rightarrow 0 .
\end{equation*}
This sequence is induced from the exact sequence
\begin{equation} \tag{$L^\Delta_{C/B/A}$} \label{eqn:14}
  0 \rightarrow \Omega_{P/A} \tensor_P C \rightarrow \Omega_{Q/A} \tensor_Q C \rightarrow \Omega_{Q/P} \tensor_Q C \rightarrow 0
\end{equation}
by applying $\Hom(-, J)$.  On the other hand~\eqref{eqn:14} is precisely the exact sequence used to define the exact triangle of cotangent complexes~\eqref{eqn:13}, cf.\ \cite[II.2.1]{Illusie}.  This proves the compatibility of the boundary maps, and completes the proof of Proposition~\ref{prop:bdry-map}.  \qed


\section{Obstruction classes}
\label{sec:obs}

In \cite{def-rings}, we showed that if $B$ is an $A$-algebra then for any $B$-module $J$,
\begin{equation*}
  \Ext^p(\bL_{B/A}, J) = H^p\bigl(A\uAlg/B, \uDer_A(B,J)\bigr) .
\end{equation*}
Illusie defined obstruction classes to various moduli problems in the former groups, and we have defined obstruction classes to the same moduli problems in the latter.  We shall show here that these classes agree.

\subsection{The boundary maps}
\label{sec:bdry-maps}

We continue to consider a sequence of ring homomorphisms $A \rightarrow B \xrightarrow{f} C$.  For any sheaf of $C$-modules $J$, there are boundary maps
\begin{gather}
  \pi^{B/A}_\ast \uDer_A(B,J) \rightarrow R^1 \pi^{C/B}_\ast \uDer_B(C,J) \label{eqn:6} \\
  R^1 \pi^{B/A}_\ast \uDer_A(B,J) \rightarrow R^2 \pi^{C/B}_\ast \uDer_B(C,J) . \label{eqn:7}
\end{gather}
We interpret these maps in terms of torsors and gerbes.

Suppose that $d : B \rightarrow J$ is an $A$-derivation.  Then $(f, d) : B \rightarrow C + J$ is a morphism of $A$-algebras over $C$.  Let $C'$ be the $B$-algebra $C + J$ over $C$, with the $B$-algebra structure coming from $(f, d)$.  Then $C'$ represents a $\uDer_B(C,J)$-torsor over $B\uAlg/C$, hence gives a section of $R^1 \pi^{C/B}_\ast \uDer_B(C,J)$.

\begin{lemma} \label{lem:obs-hom}
  The image of $d$ under the boundary map~\eqref{eqn:6} is represented by $C'$.
\end{lemma}

View $d$ instead as a derivation $BC \rightarrow J0$.  Then the image of $d$ under the boundary map to $R^1 \pi_\ast \uDer_A(BC,0J)$ is the torsor of lifts of $d$ to an $A$-derivation $BC \rightarrow JJ$.  This can be identified as the torsor of lifts of algebra homomorphisms
\begin{equation*}
  \xymatrix{
    & (B + J)(C + J) \ar[d] \\
    BC \ar@{-->}[ur] \ar[r] & (B + J)C .
  }
\end{equation*}
But this pushes forward via $\gamma_\ast$ to the torsor represented by $C'$.  \qed

Recall now that we can identify $H^1\bigl(A\uAlg/B, \uDer_A(B,J)\bigr)$ with the set of isomorphism classes of extensions of $B$ by $J$ as an $A$-algebra.  Let $B'$ be such an extension.  The category of all extensions $C'$ of $C$ as an $A$-algebra, equipped with an isomorphism $B' \simeq C' \fp_C B$, forms a gerbe on $B\uAlg/C$, banded by $\uDer_B(C,J)$.  Let $\omega$ be the class of this gerbe in $H^2\bigl(B\uAlg/C, \uDer_B(C,J)\bigr)$.

\begin{lemma} \label{lem:obs-alg}
  The image of $B'$ under the boundary map~\eqref{eqn:7} is $\omega$.
\end{lemma}

We can identify the class of $B'$ in $H^1\bigl(A\uAlg/B, \uDer_A(B,J)\bigr)$ with the class in $H^1\bigl(A\uAlg/BC, \uDer_A(BC, J0)\bigr)$ represented by $B'C$.  The image of this torsor in $H^2\bigl(A\uAlg/BC, \uDer_A(BC, 0J)\bigr)$ is the gerbe banded by $\uDer_A(BC,0J)$ parameterizing $\uDer_A(BC,JJ)$-torsors $Q$ whose induced $\uDer_A(BC,J0)$-torsor is represented by $B'C$.  That is the same thing as the gerbe of extensions $B'C'$ of $BC$, in which $B'$ is the fixed extension of $B$ that we started with.  In any such extension, we have $B' = C' \fp_C B$, so the class of the pushforward of this gerbe is $\omega$.  \qed

\subsection{Extensions of algebras}
\label{sec:exal}

Recall that if $B'$ is an $A$-algebra extension of $B$ by a square-zero ideal $J$ then $B'$ represents a $\uDer_A(B,J)$-torsor on the site $A\uAlg/B$.  As such, it is classified by an element in $H^1\bigl(A\uAlg/B, \uDer_A(B,J)\bigr)$.  In \cite{def-rings}, we have seen that there is an isomorphism
\begin{equation*}
  H^1(A\uAlg/B, \uDer_A(B,J)) \simeq \Ext^1(\bL_{B/A}, J) .
\end{equation*}
Illusie shows that $B'$ is also described up to isomorphism by an element of the latter group, $\Ext^1(\bL_{B/A}, J)$.  We will prove that these elements agree under the identification above.

We recall Illusie's construction of the class association to $B'$.  Let $P$ be a simplicial resolution of $B$ by free $A$-algebras.  Then $P' := B' \fp_B P$ is a square-zero extension of $P$ by the ideal $J$.  Since $P$ is term-by-term a free $A$-algebra it satisfies Illusie's Condition L \cite[III.1.1.7]{Illusie}, so the sequence
\begin{equation*}
  0 \rightarrow J \rightarrow \Omega_{P'/A} \tensor_P P' \rightarrow \Omega_{P/A} \rightarrow 0
\end{equation*}
is exact.  Since $\Omega_{P/A}$ is, term-by-term, a free $P$-module, and $\bL_{B/A} = \bL_{P/A} \tensor_P B$, this gives an extension
\begin{equation*}
  0 \rightarrow J \rightarrow \Omega_{P'/A} \tensor_{P'} B \rightarrow \bL_{B/A} \rightarrow 0
\end{equation*}
whose class in $\Ext^1(\bL_{B/A}, J)$ is the class of the $A$-algebra extension $B'$.

This construction defines a functor
\begin{equation*}
  \alpha : \Exal(B, J) \rightarrow \Ext^{\leq 1}(\bL_{B/A}, J)
\end{equation*}
where $\Ext^{\leq 1}(\bL_{B/A}, J)$ denotes the category of extensions of the complex $\bL_{B/A}$ by~$J$.  Recall that an extension of $\bL_{B/A}$ by $J$ is an extension $F$ of $\bL_{B/A}^0$ by $J$, and a map $\bL_{B/A}^{-1} \rightarrow F$ lifting the differential $\bL_{B/A}^{-1} \rightarrow \bL_{B/A}^0$ whose composition with $\bL_{B/A}^{-2} \rightarrow \bL_{B/A}^{-1}$ is the zero map $\bL_{B/A}^{-2} \rightarrow F$ (cf.\ \cite[1.8]{Gr}, applied to the additively cofibered category of extensions).

Both $\Exal(B,J)$ and $\Ext^{\leq 1}(\bL_{B/A},J)$ vary contravariantly with the $A$-algebra $B$, and Illusie shows that $\alpha$ is compatible with this variation.  Therefore $\alpha$ can be viewed as an equivalence of stacks
\begin{equation*}
  \alpha : \uExal(B,J) \rightarrow \uExt^{\leq 1}(\bL_{B/A}, J)
\end{equation*}
on $A\uAlg/B$.  We must show that this morphism agrees with the morphism described in \cite[Section~7]{def-rings}.

Since both $\uExal(B,J)$ and $\uExt^{\leq 1}(\bL_{B/A},J)$ are stacks on $A\uAlg/B$, this is a local question.  Any extension of $B$ by $J$ is locally isomorphic to $B + J$ in $A\uAlg/B$, so it is sufficient to show that $\alpha$ agrees with the construction of \cite{def-rings} when $B' = B + J$.  We now only have to check that $\alpha$ defines the same map on automorphisms of $B + J$ as an extension of $J$.

If $\varphi$ is an automorphism of $B + J$ then we obtain a derivation of $B$ into $J$ by composing
\begin{equation*}
  B \rightarrow B + J \xrightarrow{\varphi} B + J \rightarrow J .
\end{equation*}
This identifies the automorphism group of $B + J$ with $\Der_A(B,J)$, which is identified in \cite{def-rings} with $\Hom(\bL_{B/A}, J) = \Hom(\Omega_{B/A}, J)$ by viewing $\Omega_{B/A}$ as the universal $A$-derivation of $B$.  At this point, we only need to check that Illusie's map $\alpha$ also identifies the automorphism group of $B + J$ with $\Hom(\Omega_{B/A}, J)$ by the universal property.

Using Illusie's construction, we get an automorphism of the extension $\Omega_{B + J}$ of $\Omega_B$ by $J$.  We can compute the corresponding element of $\Hom_B(\Omega_B, J)$ as the composition of
\begin{equation*}
  \Omega_B \rightarrow \Omega_{B + J} \xrightarrow{\Omega_{\varphi}} \Omega_{B + J} \rightarrow J .
\end{equation*}
As the diagram
\begin{equation*}
  \xymatrix{
    B \ar[r] \ar[d]_d & B + J \ar[r] \ar[d]_d & B + J \ar[r] \ar[d]_d & J \ar@{=}[d] \\
    \Omega_B \ar[r] & \Omega_{B + J} \ar[r] & \Omega_{B + J} \ar[r] & J
  }
\end{equation*}
commutes, the map $\alpha$ does indeed carry $\varphi$ to the factorization of the corresponding derivation through the universal derivation $B \rightarrow \Omega_{B/A}$.

This implies
\begin{proposition}
  The classes of the extension $B'$ in the groups  $\Ext^1(\bL_{B/A}, J)$ and $H^1\bigl(A\uAlg/B, \uDer_A(B,J)\bigr)$, as defined in \cite[Th\'eor\`eme~1.2.3]{Illusie} and \cite[Section~6]{def-rings}, respectively, agree via the identification of \cite[Theorem~3]{def-rings}.
\end{proposition}

\subsection{The obstruction to deforming an algebra homomorphism}

We consider the following deformation problem:  suppose that $B$ is an $A$-algebra and $B'$ is a square-zero extension of $B$ with ideal $J$.  Find a section of the homomorphism of $A$-algebras $B' \rightarrow B$.

Note that this problem is equivalent to \cite[Probl\`eme~2.2.1.2]{Illusie}.  In the notation of loc.\ cit., that problem can be reduced to this one by setting $B' = C \fp_{C_0} B$, replacing $B$ by $v \ast B$, and replacing $J$ by $K$.

After the reduction above, Illusie's obstruction class is simply the class of the extension $B'$ of $B$ in $\Ext^1(\bL_{B/A},J)$, and we have already seen in Section~\ref{sec:exal} that this class agrees with the one defined in \cite{def-rings}.

\subsection{The obstruction to deforming an algebra}

Suppose that $B \rightarrow C$ is a homomorphism of $A$-algebras and $IJ$ is a $BC$-module.  Consider the problem of extending a square-zero extension $B'$ of $B$ by the ideal $I$ to an extension $C'$ of $C$ by $J$, fitting into a commutative diagram
\begin{equation*}
  \xymatrix{
    B' \ar@{-->}[r] \ar[d] &  C' \ar@{-->}[d] \\
    B \ar[r] & C .
  }
\end{equation*}

In \cite[Theorem~2]{def-rings}, we described an obstruction to this deformation problem lying in $H^2(B\uAlg/C, \uDer_B(C,J))$.  In order to compare this to Illusie's obstruction in $\Ext^2(\bL_{C/B}, J)$, we shall reinterpret the obstruction of \cite{def-rings} using the boundary map
\begin{equation} \label{eqn:8}
  H^1\bigl(B'\uAlg/B, \uDer_{B'}(B,J)\bigr) \rightarrow H^2\bigl(B\uAlg/C, \uDer_B(C,J)\bigr) .
\end{equation}
in the transitivity sequence associated to the sequence of ring homomorphisms $B' \rightarrow B \rightarrow C$.

If $B''$ is an extension of $B$ as a $B'$-algebra, and $J$ is the kernel of $B'' \rightarrow B$, then the compatibility of the map $B' \rightarrow B''$ with the projection to $B$ gives a homomorphism of $B$-modules~$I \rightarrow J$.  This determines an isomorphism between $H^1\bigl(B'\uAlg/B, \uDer_{B'}(B,J)\bigr)$ and $\Hom_B(I,J)$.  We therefore obtain a map
\begin{equation*}
  \Hom_B(I,J) \rightarrow H^2\bigl(B\uAlg/C, \uDer_B(C,J)\bigr) .
\end{equation*}
The map~\eqref{eqn:8} takes an extension $B''$ of $B$ by $J$ to the gerbe of extensions $C'$ of $C$ by $J$ that induce $B''$.  We have seen in Lemma~\ref{lem:obs-alg} that if $B''$ is the extension of $B$ by $J$ associated to a homomorphism $\varphi$, then this gerbe is the obstruction to the deformation problem defined in \cite[Theorem~2]{def-rings}.  We have proved
\begin{proposition}
  The obstruction $\omega$ defined in \cite[Theorem~2]{def-rings} is the image of $\varphi$ under the map
  \begin{equation*}
    \Hom_B(I,J) \rightarrow H^2\bigl(B\uAlg/C, \uDer_B(C,J)\bigr)
  \end{equation*}
  defined above.
\end{proposition}

\begin{cortoprop}
  The obstruction $\omega$ defined in \cite[Theorem~2]{def-rings} agrees with the one defined by Illusie \cite[Proposition~2.1.2.3]{Illusie} when the cotangent complex is used to represent the cohomology of $\uDer$.
\end{cortoprop}

We need only remark that the map
\begin{equation*}
  H^1\bigl(B'\uAlg/B, \uDer_{B'}(B, J)\bigr) \rightarrow H^2\bigl(B\uAlg/C, \uDer_B(C,J)\bigr)
\end{equation*}
is represented by the map of cotangent complexes $\bL_{C/B}[-2] \rightarrow \bL_{B/B'}[-1]$ as proved in Proposition~\ref{prop:bdry-map}.  The image of $\varphi$ under this map is Illusie's definition of the obstruction $\omega$ \cite[III.2.1.2]{Illusie}.  \qed

\section{Base change}
\label{sec:base-change}

Suppose that $A \rightarrow \oA$ is a homomorphism of sheaves of rings, and $B$ is an $A$-algebra.  Let $\oB = B \tensor_A \oA$.  The functor
\begin{equation*}
  \varphi_! : A\uAlg/B \rightarrow \oA\uAlg/\oB : C \mapsto C \tensor_A \oA
\end{equation*}
has the right adjoint $\varphi^\ast(\oC) = \oC \fp_{\oB} B$.  Since $\varphi_!$ is cocontinuous (and $\varphi^\ast$ continuous), we get a morphism of sites
\begin{equation*}
  \varphi : A\uAlg/B \rightarrow \oA\uAlg/\oB.
\end{equation*}
If $J$ is a sheaf of $\oB$-modules, then
\begin{equation*}
  \varphi^\ast \uDer_{\oA}(\oB, J) = \uDer_A(B, J) .
\end{equation*}
By adjunction we therefore obtain a map $\uDer_{\oA}(\oB, J) \rightarrow R \varphi_\ast \uDer_A(B, J)$, which induces a map
\begin{equation} \label{eqn:16}
  R \pi^{\oB/\oA}_\ast \uDer_{\oA}(\oB, J) \rightarrow R \pi^{B/A}_\ast \uDer_A(B, J) .
\end{equation}
The object of this section is to prove
\begin{theorem} \label{thm:base-change}
  If $\Tor^A_q(\oA, B)$ vanishes for $0 < q \leq n$ then the map
  \begin{equation} \label{eqn:17}
    R^p \pi^{\oB/\oA}_\ast \uDer_{\oA}(\oB, J) \rightarrow R^p \pi^{B/A}_\ast \uDer_A(B,J)
  \end{equation}
  is an isomorphism for $p \leq n$.
\end{theorem}

The first step is to prove
\begin{corolemma} \label{corolem:free}
  If $B$ is a free $A$-algebra, then~\eqref{eqn:16} is an isomorphism.
\end{corolemma}

Suppose that $B = A[S]$ for some sheaf of sets $S$ on $E$.  By \cite[Section~8]{def-rings}.  Both source and target of~\eqref{eqn:16} can be identified with $R \pi_\ast \pi^\ast J$ where $\pi : E/S \rightarrow E$ is the projection.  \qed

\begin{corolemma} \label{corolem:small-n}
  The theorem holds for $n = 0, 1$.
\end{corolemma}

For $n = 0$, we must show that any $A$-derivation $B \rightarrow J$ factors uniquely through an $\oA$-derivation $\oB \rightarrow J$.  This is immediate from the universal properties of the tensor product and fiber product, which identify
\begin{multline*}
  \Der_A(B, J) = \Hom^A_B(B, B + J) = \Hom^A_{\oB}(B, \oB + J) \\= \Hom^{\oA}_{\oB}(\oB, \oB + J) = \Der_{\oA}(\oB, J) .
\end{multline*}

For $n = 1$, recall that $R^1 \pi^{B/A}_\ast \uDer_A(B,J)$ is the sheaf associated to the presheaf of isomorphism classes in $\pi^{B/A}_\ast \uExal_A(B,J)$.  Likewise, $R^1 \pi^{\oB/\oA}_\ast \uDer_{\oA}(\oB, J)$ is the associated sheaf of isomorphism classes in $\pi^{\oB/\oA}_\ast \uExal_{\oA}(\oB, J)$.  It will now suffice to prove the following stronger claim.

\begin{lemma}
  If $\Tor^A_1(B, \oA) = 0$ then the natural map
  \begin{equation*}
    \pi^{\oB/\oA}_\ast \uExal_{\oA}(\oB, J) \rightarrow \pi^{B/A}_\ast \uExal_A(B,J)
  \end{equation*}
  is an equivalence of stacks over $E$.
\end{lemma}

It is equivalent to show that if $B'$ is an $A$-algebra extension of $B$ with ideal $J$ then there is an extension $\oB'$ of $\oB$ as an $\oA$-algebra, with ideal $J$, and an $A \rightarrow \oA$ morphism of extensions,
\begin{equation} \label{eqn:18}
  \xymatrix{
    0 \ar[r] & J \ar[r] \ar@{=}[d] & B' \ar[r] \ar[d] & B \ar[r] \ar[d] & 0 \\
    0 \ar[r] & J \ar[r] & \oB' \ar[r] & \oB \ar[r] & 0 .
  }
\end{equation}
To construct the second row, first note that the sequence
\begin{equation*}
  \xymatrix{
    0 \ar[r] &  J \tensor_A \oA \ar[r] &  B' \tensor_A \oA \ar[r] & \oB \ar[r] & 0
  }
\end{equation*}
is exact because $\Tor^A_1(B, \oA) = 0$.  Pushing this sequence out by the canonical map $J \tensor_A \oA \rightarrow J$ coming from the $\oA$-module structure on $J$ gives the second row of~\eqref{eqn:18} and the commutative diagram.  \qed

\begin{lemma} \label{lem:tor}
  If $\Tor^A_q(B, \oA) = 0$ for all $q$ such that $0 < q \leq n$ then $\Tor^A_q(C_p, \oA) = 0$ whenever $0 < q \leq n - q$.
\end{lemma}

Note that $C_{q + 1} = C_q \fp_B C$.  By induction, it is therefore sufficient to prove

\begin{lemma}
  If $C \rightarrow B$ and $C' \rightarrow B$ are surjective $A$-algebra homomorphisms such that $\Tor^A_q(C, \oA) = \Tor^A_q(C', \oA) = 0$ for $0 < q \leq n - 1$ and $\Tor^A_q(B, \oA) = 0$ for $0 < q \leq n$ then $\Tor^A_q(C \fp_B C', \oA) = 0$ for $0 \leq q \leq n - 1$.
\end{lemma}

Apply $\Tor^A(-, \oA)$ to the short exact sequence
\begin{equation*}
  0 \rightarrow C \fp_B C' \rightarrow C \times C' \rightarrow B \rightarrow 0 .
\end{equation*}
The long exact sequence, together with the vanishing of $\Tor^A_q(C' \times C', \oA)$ in the range $0 < q \leq n$ implies that $\Tor^A_q(C \fp_B C', \oA) = \Tor^A_{q + 1}(B, \oA)$, from which the conclusion is immediate.  \qed

\begin{proof}[Proof of Theorem~\ref{thm:base-change}]
  Let $C \rightarrow B$ be a surjective map from a free $A$-algebra to $B$.  Let $C_p$ be the $(p+1)$-fold fiber product of $C$ over $B$ (so $C_0 = C$ and $C_1 = C \fp_B C$, etc.).

  Consider the \v{C}ech spectral sequence associated to $\uDer_A(B, J)$ and this cover.  Its $E_1$-term is
  \begin{equation*}
    R^p \pi^{B/A}_\ast \uDer_A(C_q, J) \Rightarrow R^{p + q} \pi^{B/A}_\ast \uDer_A(B, J) .
  \end{equation*}
  Let $\oC = C \tensor_A \oA$ and let $\oC_q$ the $(q + 1)$-fold fiber product of $\oC$ over $\oB$.  Letting $\oE$ denote the spectral sequence
  \begin{equation*}
    R^p \pi^{\oB/\oA}_\ast \uDer_{\oA}(\oC_q, J)\Rightarrow R^{p+q} \pi^{\oB/\oA}_\ast \uDer_{\oA}(\oB,J)
  \end{equation*}
  we have a morphism of spectral sequences $\oE \rightarrow E$ abutting to the morphism~\eqref{eqn:16}.  To prove that~\eqref{eqn:17} is an isomorphism for $p \leq n$, it therefore suffices to prove that $\oE_1 \rightarrow E_1$ is an isomorphism for $p + q \leq n$; we prove this by induction on $n$.  The theorem is already proved if $n = 0, 1$, so assume that the theorem holds for some~$n$.  Then the maps
  \begin{equation} \label{eqn:19}
    R^p \pi^{\oB/\oA}_\ast \uDer_{\oA}(\oC_q, J) \rightarrow R^p \pi^{B/A}_\ast\uDer_A(C_q,J)
  \end{equation}
  are isomorphisms if
  \begin{enumerate}
  \item $q = 0$ (by Lemma~\ref{corolem:free}),
  \item $p + q \leq n + 1$ and $p \leq n$ (by the inductive hypothesis and Lemma~\ref{lem:tor}).
  \end{enumerate}
  These imply that~\eqref{eqn:19} is an isomorphism for $p + q \leq n + 1$, which in turn implies that~\eqref{eqn:17} is an isomorphism for $p \leq n + 1$.    
\end{proof}

\bibliographystyle{alpha}
\bibliography{def-rings}

\end{document}